\documentclass[12pt,twoside]{amsart}
\usepackage{amsmath}
\usepackage{amsfonts}
\usepackage{amssymb}
\usepackage{amsthm}
\usepackage[english]{babel}

\newtheorem{theo}{Theorem}
\newtheorem{lem}[theo]{Lemma}
\newtheorem{prop}[theo]{Proposition}

\newtheorem{defn}{Definition}

\newcommand{\E}{\ensuremath{\mathbb {E}}}
\renewcommand{\P}{\ensuremath{\mathbb {P}}}

\newcommand{\N}{{\mathbb N}}

\newcommand{\DD}{{\mathcal D}}

\newcommand{\A}{{\mathcal A}}
\newcommand{\B}{{\mathcal B}}
\newcommand{\F}{{\mathcal F}}

\def\E{\mathbb{E}}

\def\F{{\mathcal F}}
\def\R{\mathbb{R}}

\begin{document}

\title
[Almost everywhere of  series]
{Study of almost everywhere convergence of  series by means of martingale methods }

\author{Christophe Cuny }
\address{Laboratoire MAS, Ecole Centrale de Paris, Grande Voie des Vignes,
92295 Chatenay-Malabry cedex, FRANCE}
\email{christophe.cuny@ecp.fr}
\author{Ai Hua Fan}
\address{Fan Ai-Hua: LAMFA UMR 7352, CNRS \\
	Facult\'e des Sciences \\
	Universit\'e de Picardie Jules Verne \\
	33, rue Saint Leu \\
	80039 Amiens CEDEX 1, France}
\email{ai-hua.fan@u-picardie.fr}
\dedicatory{}

\subjclass{Primary: ; Secondary: }

\begin{abstract} Martingale methods are  used to study the almost everywhere convergence
of general function series. Applications are given to ergodic series, which improves recent results of Fan \cite{FanETDS},  and to dilated series, including Davenport series, which
completes results of Gaposhkin \cite{Gaposhkin67} (see also \cite{Gaposhkin68}). Application is also given to the almost everywhere convergence with respect to Riesz products of lacunary series.
  \end{abstract}

\maketitle

\section{Introduction}
Let us first state two problems which motivate the investigation in this paper. One comes from the classical
analysis and the other from the ergodic theory.

Given a function $f\in L^1(\mathbb{R}/\mathbb{Z})$
(or equivalently a locally integrable $1$-periodic function on the real line $\mathbb{R}$) such that $\int_0^1 f(x) d x=0$, an increasing sequence of
positive integers $(n_k)_{k\ge 0}\subset \mathbb{N}$ and a sequence of complex numbers $(a_k)_{k\ge 0}\subset \mathbb{C}$, one would like to investigate the convergence (almost everywhere convergence or $L^1$-convergence etc) of the following series, called {\em dilated series},
\begin{equation} \label{DS}
     \sum_{k=0}^\infty a_k f(n_k x).
\end{equation}
If the series 
converges almost everywhere (a.e.) whence $(a_k)\in \ell^2$, we say that
$\{f(n_k x)\}$ is a {\em convergence system}.  The famous Carleson theorem states that
both $\{\sin n x\}$ and $\{\cos n x\}$ are convergence systems. In general, one should find suitable conditions
on $f$ and on $(n_k)$ for $\{f(n_k x)\}$ to be a convergence system. This is a long standing problem. One may consult the survey by Berkes and Weber \cite{BW} and, for recent progresses, Weber \cite{Weber} and the references
therein.

The other problem is the almost everywhere convergence  of the so-called {\em ergodic series}
\begin{equation}\label{ES0}
     \sum_{k=0}^\infty a_k f(T^k x)
\end{equation}
where $T$ is a measure-preserving map on a probability space $(X, \mathcal{B}, \mu)$ and $f\in L^1(\mu)$ such that $\int f d\mu =0$. To be more precise, we want to find
sufficient conditions on the dynamical system $(X, \mathcal{B}, \mu, T)$
and/or on $f$, such that the series
\eqref{ES0} converges for \emph{any} sequence $(a_n)_{n\in \N}\in \ell^p$,
where $1\le p\le 2$ is fixed (our main interest is in the case $p=2$).
Sufficient conditions for the a.e. convergence with specific (regular) sequences $(a_n)_{n\in \N}$ may be found for instance in \cite{Cuny-a.e.}.
The fact that conditions have to be imposed to ensure the a.e. convergence
of \eqref{ES0} may be illustrated by the following result of Dowker and Erd\"{o}s \cite{DE}:  if $\mu$ is an non-atomic ergodic measure, for every positive sequence $(a_n)_{n\in \N}$ such that $\sum_{k\in \N} a_k=\infty$, there exists $f\in L^\infty(\mu)$
with $\int f d\mu=0$ such that $\sum_{k\in \N} a_k f(T^k x)$ diverges a.e. Some sufficient conditions are recently found in \cite{FanETDS}
for $\{f\circ T^n\}$ be to be a convergence system.

\medskip

In this paper, we will study series in a general setting. The above situations are special examples.  Let $(\Omega, \A, \mathbb{P})$ be a probability space. Let $(Z_n)_{n\ge 0}\subset L^p(\Omega,\A,\P)$, $p\ge 1$,  be a
sequence of $L^p$-integrable random variables with $\mathbb{E} Z_n =0$. We shall study the almost sure convergence of the
random series
\begin{equation}
\sum_{n=0}^\infty Z_n(\omega).
\end{equation}

Suppose that we are given an increasing  filtration $(\A_n)_{n\in \N}$ such that $\A_0=\{\emptyset,\Omega\}$ and $\A_\infty= \A$ where
$\A_\infty:=\bigvee_{n=0}^\infty \A_n$.  We will analyze random variables
using this filtration.  For every $n\in \N\cup\{\infty\}$, denote $\E^n:= \E(\cdot|\A_n)$ and introduce the operator
$$\DD_n=\E^{n+1}-\E^{n}.
$$
If $Z\in L^p(\Omega,\A,\P)$ with $\E Z=0$ ($p\ge 1$), we have the decomposition
\begin{equation*}\label{analysis}
Z=\sum_{n=0}^\infty \DD_{n}Z \, ,
\end{equation*}
where the convergence takes place a.e. and in $L^p$. This is what we mean by the analysis based on the filtration $(\A_n)_{n\in \N}$ and it is a simple consequence of Doob's convergence theorem of martingales.

\marginpar{change}One of our main results is the following theorem, which is a special case of Theorem \ref{theo-gen} corresponding to $p=2$.

\smallskip

\noindent {\bf Theorem A.}
{\it
Let $(Z_n)_{n\in \N }\subset L^2(\Omega,\A,\P)$ be such that
 $\E Z_n=0$ for every $n\in \N$. Then the series $\sum_{n\in \N}Z_n $
converges $\P$-a.s. and in $L^2(\Omega,\A,\P)$ under the following set of conditions
\begin{equation}\label{series1} \sum_{k=0}^\infty\Big( \sum_{n=0}^\infty \|\DD_{n+k}
Z_{n}\|_2^2\Big)^{1/2}<\infty
\end{equation} and
\begin{equation}\label{series2}
\sum_{k=1}^\infty\Big( \sum_{n =0}^\infty \|\DD_n
Z_{n+k}\|_2^{2}\Big)^{1/2}<\infty\,.
\end{equation}
}

One of ingredients in the proof of Theorem A is the Doob maximal inequality of martingales.
Assume that $\{Z_n\}\subset L^2(\Omega, \mathcal{A}, P)$ are independent and
$\mathbb{E} Z_n =0$ for all $n$.
It follows from Theorem A that $\sum \mathbb{E}|Z_n|^2<\infty$ implies the almost sure convergence of $\sum Z_n$. This is a trivial application of Theorem A, because this known result of Kolmogorov is actually covered by the Doob convergence theorem of martingales.

We can also make an analysis using a decreasing filtration.
Suppose that we are given a decreasing  filtration $(\B_n)_{n\in \N}$ such that
$\B_0=\A$.  Let $\B_\infty:=\bigcap_{n=0}^\infty \B_n$.  For every $n\in \N\cup\{\infty\}$, denote $\mathbb{E}_n =\mathbb{E} (\cdot | \mathcal{B}_n)$
and introduce
$$\mathfrak{d}_n =\E_{n}-\E_{n+1}.
$$
Let $p\ge 1$.
For
$Z\in L^p(\Omega,\A,\P)$ with $\E (Z|\B_\infty)=0$ which implies $\mathbb{E} Z=0$,   we have the decomposition
\begin{equation*}\label{analysis2}
Z=\sum_{n=0}^\infty \mathfrak{d}_{n}Z \, ,
\end{equation*}
where the series converges in $L^p$ and $\P$-a.s.
\medskip

\noindent{\bf Theorem B.}
Let $(Z_n)_{n\in \N }\subset L^2(\Omega,\A,\P)$ be such that
 $\E_\infty (Z_n)=0$ for every $n\in \N$. Then the series $\sum_{n\in \N}g_n $
converges $\P$-a.s. and in $L^2(\Omega,\A,\P)$ under the following conditions
\begin{equation}\label{series1+} \sum_{k=0}^\infty\Big( \sum_{n=0}^\infty \|\mathfrak{d}_{n+k}
Z_{n}\|_2^2\Big)^{1/2}<\infty
\end{equation} and
\begin{equation}\label{series2+}
\sum_{k=1}^\infty\Big( \sum_{n =0}^\infty \|\mathfrak{d}_n
Z_{n+k}\|_2^2\Big)^{1/2}<\infty\,.
\end{equation}

 As an application of Theorem A, we have the following  theorem, in which the condition (\ref{condition0}) is sharp (see Proposition \ref{example-gaposhkin-dyn}).

 Recall first that a sequence of positive integers
 $(n_k)_{k\in \N}$ is said to be {\em  Hadamard lacunary} if $\inf_k n_{k+1}/n_k\ge q >1$ for some $q$ and that the {\em modulus of $L^2$-continuity} of $f \in L^2(\mathbb{R}/\mathbb{Z})$ is defined by
 $\omega_2
 (\delta, f) = \sup_{0\le h \le \delta}\|f(\cdot + h)- f(\cdot)\|_2$.

\smallskip

\noindent {\bf Theorem C.}
{\it
Suppose that  $f\in L^2(\mathbb{R}/\mathbb{Z})$ satisfies $\int f(x) d x=0$ and
\begin{equation}\label{condition0}
\sum_{n\in \N} \frac{\omega_2(2^{-n},f)}{\sqrt{n}}<\infty
\end{equation}
and that
$(n_k)_{k\in \N}$ is Hadamard lacunary. Then
$\{f(n_k x)\}$ is a convergence system.
}
\smallskip

Let us look at a very interesting special system $\{f(n_k x)\}$ where $f$ is a Davenport function. Let $\lambda >0$. The function
$$
f_\lambda(x) = \sum_{m=1}^\infty \frac{\sin 2\pi m x}{m^\lambda}
$$
is well defined because the series converges everywhere. It is  called Davenport function.
When $\lambda >1/2$, it is $L^2$-integrable and
$\{f_\lambda(nx)\}_{n\ge 1}$ is a complete system in $L^2([0, 1])$ (\cite{Wintner}). When $\lambda >1$, $\{f_\lambda(nx)\}_{n\ge 1}$ is even
a Riesz basis  (\cite{HLS, LS}). 
 When $1/2<\lambda \le 1$, $\{f_\lambda(nx)\}_{n\ge 1}$ is not a Riesz basis, but Br\'emont proved that $\{f_\lambda(n_kx)\}_{n_k\ge 1}$
is a Riesz sequence (see Definition \ref{Riesz}) for any Hadamard lacunary sequence $\{n_k\}$.

\medskip

\noindent {\bf Theorem D}
{\it
	Let $\lambda >1/2$. Suppose that $\{n_k\}\subset \mathbb{N}$ is lacunary in the sense of Hadamard.
	Then the so-called Davenport series $\sum_{k=1}^\infty a_k f_\lambda(n_k x)$
	converges almost everywhere if ond only if $\sum_{k=1}^\infty |a_k|^2 <\infty$.
}

\smallskip
In this special case of Davenport series,  the condition
$\sum_{k=1}^\infty |a_k|^2 <\infty$ is also proved necessary for the almost everywhere convergence. Both the sufficiency and necessity are new.
 \medskip

Now let us give an application of Theorem B. Consider a measure-theoretic dynamical system
$(X, \mathcal{B}, \mu T)$.
Let $L$ be the associated transfer (or Perron-Frobenius) operator, which is defined by
 \begin{equation}\label{perron-frobenius}
     \int f \cdot Lg d\mu = \int f\circ T \cdot g d\mu \qquad (\forall f\in L^\infty(\mu), \forall g \in L^1(\mu)).
 \end{equation}
 As an application of Theorem B, we have the following theorem, in which the condition (H2) is
sharp to some extent (see Proposition \ref{example-gaposhkin}).
\medskip

\noindent {\bf Theorem E}
{\it
Assume that $(X, \mathcal{B}, T, \mu)$ is an ergodic measure-preserving dynamical system and $f \in L^2(\mu)$. Suppose
      \\
      \ \indent {\rm (H1)}\ \ $
      \lim_{m\to \infty} \mathbb{E} (f|T^{-m}\mathcal{B})=0$ a.e., which implies $\int f d\mu=0$; \\
      \ \indent {\rm (H2)} \ $
          \sum_{n=1}^\infty \frac{\|L^n f\|_2}{\sqrt n } <\infty
      $\\
      Then $\{f(T^n x)\}$ is a convergence system.
      }
      \medskip

This theorem improves a result in \cite{FanETDS}, where the norm $\|\cdot\|_\infty$ was used instead of the norm $\|\cdot\|_2$.
\medskip

The paper is organized as follows. In Section \ref{SectIncreasing} we performs an analysis using increasing filtrations. Theorem A will be proved there, together with some
more general results. Conditions in Theorem A will be converted into some more practical conditions and divergence will also discussed.    Section \ref{SectDecreasing} is parallel to Section \ref{SectIncreasing}, as Theorem B is parallel to Theorem A. There an analysis is made using
decreasing filtration. But details are omitted and some details
can also be found in \cite{FanETDS}. Application of Theorem A to
ergodic series is discussed in Section \ref{SectES} and Theorem
E will be proved there. Section \ref{SectDS} is devoted to dilated series (Theorem C and Theorem D) and Section \ref{SectRP} is devoted to lacunary series and their
almost everywhere convergence with respect to Riesz product and
to more general inhomogeneous equilibriums.

\section{Analysis using increasing filtration}\label{SectIncreasing}

Let $(\Omega, \A, \mathbb{P})$ be a probability space and $p\ge 1$. Let $(Z_n)_{n\ge 0}\subset L^p(\Omega,\A,\P)$ be a
sequence  random variables such that $\mathbb{E}Z_n =0$ for all n. We shall study the almost sure convergence of the
random series $\sum_{n=0}^\infty Z_n(\omega)$.
For $n\ge 0$, denote the partial
sum $$
S_n = \sum_{k=0}^n Z_k.
$$
Then  define the maximal functions
$$
   S^*(\omega) =\sup_{n\in \N} |S_n(\omega)|, \quad    S_N^*(\omega) =\max_{0\le n \le N} |S_n(\omega)| \ \ (\forall N \ge 0).
$$
In this section, we give an analysis of the series by using an increasing filtration. In the next section, we do the same by using an decreasing filtration. Notice that there will be minor differences
between two analysis. But applications will show that there is a question of how to choose a filtration. For example, for the dilated series, we will introduce an increasing filtration and for the ergodic series, there is a natural decreasing filtration.

\subsection{Decomposition relative to an increasing filtration}
Suppose that we are given an increasing  filtration $(\A_n)_{n\in \N}$. Assume that
$\A_0=\{\emptyset,\Omega\}$ and $\A_\infty= \A$ where
$\A_\infty:=\bigvee_{n=0}^\infty \A_n$. 
For every $n\in \N\cup\{\infty\}$, denote $\E^n:= \E(\cdot|\A_n)$ and
$$\DD_n=\E^{n+1}-\E^{n}.
$$
The operators $\DD_n$ have the following remarkable properties.
The proofs, which are easy,  are left to the reader.


\begin{lem} \label{martingale}
 Assume
$h\in  L^1(\Omega, \A, \P)$ and $f, g\in L^2(\Omega, \A, \P)$.
The operators $\DD_n$ have the following properties.\\
\indent {\rm (1)}\
For any $n\ge 0$, $\DD_n$ is $\A_{n+1}$-measurable and $\E^{n}\DD_{n} f=0$.\\
\indent {\rm (2)}\
For any distinct integers $n$ and $m$,
$\DD_n f$ and $\DD_m g$ are orthogonal.\\
\indent {\rm (3)}\ For  any $N_1<N_2$ we have
$$
     \sum_{n=N_1}^{N_2} \|\DD_n f\|_2^2 = \|\E^{N_2 +1} f - \E^{N_1} f\|_2^2.
$$
\end{lem}


The first assertion implies that for any sequence $(f_n)\in L^1(\Omega, \A, \P)$,
$(\DD_n f_n)$ is a sequence of martingale differences. The second assertions will be referred to as the
orthogonality of the martingale difference.
\medskip

For any integral random variable of zero mean, we may decompose it into martingale as follows. In the following lemma, we include an inequality of Bahr-Esseen-Rio and an inequality of Burkholder.

\begin{lem} \label{decomposition} Let
$Z\in L^p(\Omega,\A,\P)$, $p\ge 1$ with $\E Z=0$. We have the decomposition
\begin{equation}\label{decomp0}
Z=\sum_{n=0}^\infty \DD_{n}Z \, ,
\end{equation}
where the series converges in $L^p$ and $\P$-a.s. Moreover we have
\begin{equation}\label{Rio0}
\|Z\|_p\le \max(1,\sqrt{p-1}) \Big(\sum_{n=0}^\infty \|\DD_{n}Z\|_p^{p'}
\Big)^{1/p'}\, ,
\end{equation}
where $p':=\min(2,p)$ and, if $p>1$,
\begin{equation}\label{Burkholder0}
\Big\|\Big(\sum_{n=0}^\infty |\DD_{n}Z|^2\Big)^{1/2}\Big\|_p \le C_p
\|Z\|_p\, .
\end{equation}
\end{lem}

\noindent {\bf Proof.}  By assumptions, $(\E^n Z)_{n\in
\N}$ is a uniformly integrable martingale converging in $L^p$ and $\P$-a.s. to $\E^\infty Z=Z$. Hence, \eqref{decomp0} follows
 from the equality $$\sum_{n=1}^N \DD_n Z=\E^{N+1} Z -\E^0 Z
$$
where $\E^0 Z = \E(Z|\A_0) =\E Z=0$, for $\A_0=\{\emptyset, \Omega\}$.
Then, \eqref{Rio0} follows from von Bahr-Esseen \cite{BE} when
$1\le p\le 2$ and from Rio \cite{Rio} when $p\ge 2$, while \eqref{Burkholder0}
is the Burkholder inequality. $\Box$
\medskip

\noindent {\bf Remark.} The inequality \eqref{Burkholder0} is nothing but
the (reverse) Burkholder inequality. In particular, we have a converse inequality but we shall only need \eqref{Burkholder0} in the sequel.

\smallskip
We call $\DD_n Z$ the {\rm $n$-th order detail} of $Z$ with respect to $(\A_n)$.

\medskip

For any $Z_n$, we  have the following decomposition
$$
    Z_n = X_n + Y_n \quad \mbox{\rm with} \ \ X_n= Z_n -\E^n Z_n, \ \ Y_n = \E^n Z_n.
$$
Notice that $\E^n X_n =0$ i.e. $X_n$ is conditionally centered,  and $Y_n$ is $\A_n$-measurable \marginpar{rajout}with
$\E(Y_n)=0$. Our random series is thus decomposed into
two random series
$$
     \sum_{n=0}^\infty Z_n = \sum_{n=0}^\infty X_n + \sum_{n=0}^\infty Y_n.
$$
The convergence of $\sum Z_n$ is reduced to those of $\sum X_n$ and $\sum Y_n$. In the sequel, we separately study these two series.

\medskip

\subsection{Conditionally centered  series $\sum X_n$}

The maximal functions associated to the series $\sum X_n$
will be denoted by $ S_{N, X}^*$ and $S^*_X$. Recall that
 we write $p' =\min(2, p)$ for $p\ge 1$.

\begin{prop}\label{mart1}
Let $(X_n)_{n\in \N}\subset L^p(\Omega,\A,\P)$ be such that $\E^n(X_n)=0$ for every $n\in \N$ ($p> 1$).
Then, for every $N\in \N$, we have the following maximal
inequality
\begin{equation}
\left\| S_{N,X}^* \right\|_p
\le \frac{p}{p-1} \max(1,\sqrt{p-1}) \sum_{k=0}^\infty\Big( \sum_{\ell=0}^N \|\DD_{\ell+k} X_{\ell}\|_p^{p'}
\Big)^{1/p'} .
\end{equation}
Consequently, the series $\sum_{n\in \N}X_n $
converges $\P$-a.s. and in $L^p(\Omega,\A,\P)$ under the following condition:
\begin{equation}\label{main-cond}
    \sum_{k=0}^\infty\Big( \sum_{n=0}^\infty \|\DD_{n+k}
X_{n}\|_p^{p'}\Big)^{1/p'}<\infty\, .
\end{equation}
Moreover, when \eqref{main-cond} holds, we have
$S^*_X\in L^p(\Omega,\A,\P)$.
\end{prop}





\noindent {\bf Proof.}  By
the decomposition in Lemma \ref{decomposition} applied to each $X_\ell$ and
using the fact $\E^\ell (X_\ell)=0$ which implies 
that, for every
$k\le \ell$, $\E^k X_\ell = \E^k(\E^\ell X_\ell)=0$, hence that $\DD_k X_\ell=0$ for $k<\ell$, we have
$$X_\ell= \sum_{k\ge \ell}\DD_k X_\ell=\sum_{k=0}^\infty \DD_{k+\ell}X_\ell\, .
$$
Let $N\in \N$ and $0\le n\le N$. We obtain that
$$
|S_n|= \left|\sum_{\ell=0}^n X_\ell\right|=
\left|\sum_{\ell=0}^n \sum_{k=0}^\infty \DD_{\ell+k} X_{\ell}\right|
\le \sum_{k=0}^\infty \max_{0\le m\le N} \left|\sum_{\ell=0}^m
\DD_{\ell+k} X_{\ell}\right|\, .
$$
Now, for every $k\in \N$ fixed,  the sequence $(\sum_{\ell=0}^m\DD_{\ell+k} X_{\ell})_{m\in \N}$
is a  martingale. Hence, by Doob's maximal inequality and the von Bahr-Esseen-Rio inequality \eqref{Rio0}
in Lemma \ref{decomposition}, we have
$$
\| S_N^*\|_p
\le \frac{p}{p-1} \max(1,\sqrt{p-1}) \sum_{k=0}^\infty\Big( \sum_{\ell=0}^N \|\DD_{\ell+k} X_{\ell}\|_p^{p'}
\Big)^{1/p'}\, .
$$

Thus we have proved the
maximal inequality.
Consequently, for any $N'\le N''$, setting $K_p:=\frac{p}{p-1} \max(1,\sqrt{p-1})$, we have
\begin{equation*}\label{Estimates}
\|\max_{N'\le p,q\le N''}|S_{p} -S_{q}|\|_p\le
2\| \max_{N'\le n\le N''} |S_n - S_{N'}| \|_p
\le 2K_p \sum_{k=0}^\infty\Big( \sum_{\ell=N'+1}^{N''} \|\DD_{\ell+k} X_{\ell}\|_p^{p'}
\Big)^{1/p'}\, .
\end{equation*}
Letting $N''\to +\infty$, we infer that
\begin{equation}\label{Estimates}
\|\sup_{ p,q\ge N'}|S_{p} -S_{q}|\|_p
\le 2K_p \sum_{k=0}^\infty\Big( \sum_{\ell\ge N'+1} \|\DD_{\ell+k} X_{\ell}\|_p^{p'}
\Big)^{1/p'}<\infty\, ,
\end{equation}
by \eqref{main-cond}. By the Lebesgue dominated convergence theorem
on $\N$, applied to the counting measure, we deduce that
\begin{equation*}
\|\sup_{ p,q\ge N'}|S_{p} -S_{q}|\|_p
\underset{N'\to +\infty}\longrightarrow 0\, .
\end{equation*}
By the Fatou lemma
we infer that almost surely
$$
\sup_{ p,q\ge N'}|S_{p} -S_{q}|\underset{N'\to +\infty} \longrightarrow
0\, ,
$$
which finishes the proof. \hfill $\square$
\medskip

\noindent{\bf Remark.} If $p\ge 2$, then $p'=2$ and the condition \eqref{main-cond}
with $p\ge 2$ is stronger than  the condition \eqref{main-cond} with $p=2$. Hence,
the relevance of the proposition when $p\ge 2$ lies in the integrability
of the maximal function $S_X^*$.  While if one is only concerned with the  a.s. convergence
it is better to apply the proposition with $p=2$. However, as we shall see,  the control of $S^*$ in $L^p$ with $p>2$
will allows us to study the divergence of the series $\sum X_n$.



\subsection{Adapted series $\sum Y_n$}

The maximal functions associated to the series $\sum Y_n$
will be denoted by $ S_{N, Y}^*$ and $S^*_Y$.

\begin{prop}\label{mart2}
Let $(Y_n)_{n\ge 1 }\subset L^p(\Omega,\A,\P)$ be such that
$Y_n$ is $\A_n$-measurable  and $\E Y_n=0$ for every $n\ge 1$
( $p>1$). Then,
for every $N\ge 1$,
\begin{equation}\label{ine-max}
\| S_{N, Y}^* \|_p
\le K_p \sum_{k=1}^{N}\Big( \sum_{\ell=k}^N \|\DD_{\ell-k} Y_{\ell}\|_p
^{p'}
\Big)^{1/p'}\,
\end{equation}
Consequently, the series $\sum_{n\in \N}Y_n $
converges $\P$-a.s. and in $L^p(\Omega,\A_0,\P)$ under the following condition
 \begin{equation}\label{main-cond2}
 \sum_{k=1}^\infty\Big( \sum_{n =0}^\infty  \|\DD_{n}
Y_{n+k}\|_p^{p'}\Big)^{1/p'}<\infty.
\end{equation}
Moreover, when \eqref{main-cond2} holds, we have
$S^*_Y\in L^p(\Omega,\A,\P)$.
\end{prop}

\noindent {\bf Proof.} The proof is similar to that of Proposition \ref{mart1}.
By assumption, we infer that for every $\ell\ge 1$,
$$
Y_\ell=\sum_{k=0}^{\ell-1}\DD_k Y_\ell=\sum_{k=1}^{\ell}\DD_{\ell-k} Y_\ell\, .
$$
Hence,
$$
|\sum_{\ell=1}^n Y_\ell|
=|\sum_{\ell=1}^n \sum_{k=1}^{\ell} \DD_{\ell-k} Y_{\ell}|
\le \sum_{k=1}^{N} \max_{k\le m\le N} |\sum_{\ell=k}^m
\DD_{\ell-k} Y_{\ell}|\, .
$$
Since, for every fixed $k\in \N$, $(\sum_{\ell=k}^m
\DD_{\ell-k} Y_{\ell})_{m\ge 1}$ is a martingale, then \eqref{ine-max} follows from the Doob maximal inequality and the
Bahr-Esseen-Rio inequality \eqref{Rio0}.

\smallskip

The proof of the $\P$-a.s. and $L^p$-convergence may be done as for
Proposition \ref{mart1}. \hfill $\square$

\subsection{Convergence and integrability of general series $\sum Z_n$}

Combining Proposition \ref{mart1} and Proposition \ref{mart2} and
using the decomposition $Z_n=X_n + Y_n$, we derive the following
theorem.

\begin{theo}\label{theo-gen}
	Let $p>1$.
Let $(Z_n)_{n\in \N }\subset L^p(\Omega,\A,\P)$ be such that
 $\E Z_n=0$ for every $n\in \N$. Then,
for every $N\in \N  $,
\begin{equation}\label{ine-max-gen}
\| S_{N,Z}^* \|_p
\le K_p \Big(\sum_{k=0}^\infty\Big( \sum_{\ell=0}^N \|\DD_{\ell+k} Z_{\ell}\|_p^{p'}
\Big)^{1/p'}+ \sum_{k=1}^{N}\Big( \sum_{\ell=k}^N \|\DD_{\ell-k} Z_{\ell}
\|_p^{p'}\Big)^{1/p'}\Big)\, .
\end{equation}
Consequently, the series $\sum_{n\in \N}Z_n $
converges $\P$-a.s. and in $L^p(\Omega,\A,\P)$ under the following set of conditions
\begin{equation}\label{series1} \sum_{k=0}^\infty\Big( \sum_{n=0}^\infty \|\DD_{n+k}
Z_{n}\|_p^{p'}\Big)^{1/p'}<\infty
\end{equation} and
\begin{equation}\label{series2}
\sum_{k=1}^\infty\Big( \sum_{n =0}^\infty \|\DD_n
Z_{n+k}\|_p^{p'}\Big)^{1/p'}<\infty\,.
\end{equation}
Moreover, if both conditions \eqref{series1} and \eqref{series2} are satisfied, then $S^*_Z\in L^p(\Omega,\A,\P)$.
\end{theo}

\noindent {\bf Remark.} If $(Z_n)$ is adapted to $(\A_n)$, then the condition (\ref{series1}) is trivially satisfied.

\smallskip
\noindent {\bf Proof.} Recall that $X_n = Z_n - \E^n Z_n$ and $Y_n =\E^n Z_n$. The theorem follows immediately from
the decomposition $Z_n = X_n + Y_n$,
Proposition \ref{mart1} and Proposition \ref{mart2} and the following simple facts
$$
    S_N^* \le  S_{N, X}^* +  S_{N,Y}^*,
$$
$$
    \DD_{\ell +k} X_\ell = \E^{\ell + k + 1} (Z_\ell - \E^\ell Z_\ell)
     - \E^{\ell + k} (Z_\ell - \E^\ell Z_\ell) =  \DD_{\ell +k} Z_\ell,
$$
$$
    \DD_{\ell - k} Y_\ell = \E^{\ell - k + 1} (\E^\ell Z_\ell)
     - \E^{\ell - k} (\E^\ell Z_\ell) =  \DD_{\ell - k} Z_\ell.
$$
$\Box$
\medskip

We call (\ref{series1}) the condition on higher order details and (\ref{series2}) the condition on lower order details.

\subsection{Practical criteria}

In order to apply Theorem \ref{theo-gen}, it is sometimes more convenient to use the sufficient conditions in the next lemma.
The proof of the lemma will use the Burkholder inequality.

The condition (\ref{series1b}) suggests  that we need some order of approximation
of $Z_k$ by $\mathbb{E}(Z_k |\mathcal{A}_n)$ as $n$ tends to infinity and the condition (\ref{series2b}) suggests that  the conditional expectation $\mathbb{E}(\cdot |\mathcal{A}_n))$ would contract in some order on the space $L^p_0(\Omega, \mathcal{A}, P)$ consisting of
$p$-integrable variables with zero mean; sometimes it is really the case
(see Lemma \ref{contraction}, see also Theorem 2 in \cite{F1999} and Lemma 2 in \cite{FP}).

\begin{lem}\label{lemme}
Let $(Z_n)_{n\in \N }\subset L^p(\Omega,\A,\P)$, $p>1$, with $\E Z_n =0$ for all $n$. The condition   \eqref{series1}
on the higher order details
is satisfied  if
\begin{equation}\label{series1b}
\sum_{\ell=0}^\infty 2^{\ell(1-1/p)}\Big( \sum_{k=0}^\infty \|Z_k-\E^{2^\ell+k-1}Z_k\|_p^{p'}\Big)^{1/p'}<\infty\,.
\end{equation}
The condition \eqref{series2} on the lower order details is satisfied
if
\begin{equation}\label{series2b}
\sum_{\ell=0}^\infty 2^{\ell(1-1/p)}\Big( \sum_{k=2^\ell}^\infty
\|\E^{k+1-2^\ell}Z_k\|_p^{p'}\Big)^{1/p'}<\infty\, .
\end{equation}
\end{lem}

\noindent{\bf Proof.} We make the proof when $p\ge 2$, then $p'=2$. The proof when
$1<p<2$ may be done similarly. Look first at the condition (\ref{series1}) on the higher order of details.
Cutting the sum over $k$ into dyadic blocks and applying Cauchy-Schwarz inequality to each block  we get
$$
\sum_{k=0}^\infty\Big( \sum_{n=0}^\infty \|\DD_{n+k}
Z_{n}\|_p^2\Big)^{1/2}\le \sum_{\ell=0}^\infty
2^{\ell/2} \Big(\sum_{k=2^{\ell}-1}^{2^{\ell+1}-2} \sum_{n=0}^\infty\|\DD_{n+k}
Z_{n}\|_p^2\Big)^{1/2}\, .
$$
Now, using successively  the H\"older inequality (notice that $p/2\ge 1$), the trivial inequality
$\|\cdot \|_{\ell^p}\le \|\cdot \|_{\ell^2}$ and the reverse Burkholder inequality \eqref{Burkholder0}, we get that
\begin{eqnarray*}
\sum_{k=2^{\ell}-1}^{2^{\ell+1}-2}\|\DD_{n+k}Z_{n}\|_p^2
& \le &
2^{\ell (1-2/p)}\Big(\sum_{k=2^{\ell}-1}^{2^{\ell+1}-2}\|\DD_{n+k}Z_{n}\|_p^p\Big)^{2/p}\\
&= & 2^{\ell (1-2/p)}\Big(\E\big(\sum_{k=2^{\ell}-1}^{2^{\ell+1}-2}(\DD_{n+k}Z_{n})^p
\big)\, \Big)^{2/p}\\
&\le& 2^{\ell (1-2/p)}\Big\|\big(\sum_{k\ge 2^{\ell}-1}
(\DD_{n+k}Z_{n})^2\big)^{1/2}\, \Big\|_p^{2}\\
&\le&
C_p2^{\ell (1-2/p)}\|Z_n -\E^{n+ 2^{\ell}-1}Z_n\|_p^2
\, .
\end{eqnarray*}
Thus the first assertion follows.
Similarly, we have
\begin{eqnarray*}
\sum_{k=1}^\infty\Big( \sum_{n =0}^\infty \|\DD_n
Z_{n+k}\|_p^2\Big)^{1/2}
&=&\sum_{k=1}^\infty\Big( \sum_{n =k}^\infty \DD_{n-k}
Z_{n}\|_p^2\Big)^{1/2}\\
&\le & \sum_{\ell=0}^\infty
2^{\ell/2} \Big( \sum_{k=2^{\ell}}^{2^{\ell+1}-1} \sum_{n =k}^\infty
\DD_{n-k}
Z_{n}\|_p^2\Big)^{1/2}\, .
\end{eqnarray*}
Now, we first change the order of summation to get
\begin{gather*}
\sum_{k=2^{\ell}}^{2^{\ell+1}-1} \sum_{n =k}^\infty
\|\DD_{n-k}Z_{n}\|_p^2  = \sum_{n=2^\ell}^{\infty}
\sum_{k=2^\ell}^{\min({n,2^{\ell+1}-1})} \|\DD_{n-k}Z_{n}\|_p^2 \,.
\end{gather*}
Then, using the same arguments as above, we have
\begin{eqnarray*}
\sum_{k=2^\ell}^{\min({n,2^{\ell+1}-1})} \|\DD_{n-k}Z_{n}\|_p^2
&\le&
C_p2^{\ell (1-2/p)} \|\E^{n+1-2^{\ell}}Z_n -\E^0Z_n\|_p^2\\
&=&
C_p2^{\ell (1-2/p)} \|\E^{n+1-2^{\ell}}Z_n \|_p^2.
\end{eqnarray*}
Finally we get
$$
\sum_{k=1}^\infty\Big( \sum_{n =0}^\infty \DD_n
Z_{n+k}\|_p^2\Big)^{1/2}
\le  C_p\sum_{\ell=0}^\infty 2^{\ell (1-1/p)}\Big( \sum_{k=2^\ell}^\infty
\|\E^{k+1-2^\ell}Z_k\|_p^2\Big)^{1/2}.
$$
$\Box$

\medskip

In view of the above results and of Lemma 5 of \cite{FanETDS}, one expects
to have better integrability properties for $S_Z^*$ when
$(Z_n)\subset L^\infty(\Omega, \A,\P)$.

\begin{theo}\label{theo-bounded}
Let $(Z_n)_{n\in \N }\subset L^\infty(\Omega,\A,\P)$  with $\E Z_n =0$ for all $n$. Assume that
\begin{equation}\label{series1-infty}
\Delta_1:=\sum_{\ell=0}^\infty \Big( \sum_{k=0}^\infty \|Z_k-\E^{\ell+k}Z_k\|_\infty^{2}\Big)^{1/2}<\infty\,.
\end{equation}
and
\begin{equation}\label{series2-infty}
\Delta_2:=\sum_{\ell=0}^\infty \Big( \sum_{k=\ell}^\infty
\|\E^{k+1-\ell}Z_k\|_\infty^{2}\Big)^{1/2}<\infty\, .
\end{equation}
Then $\sum_{n\in \N} Z_n$ converges $\P$-a.e. and in any
$L^p$ with $p\ge 1$. Moreover, we have
$$
\E({\rm e}^{\beta (S_Z^*)^2})<\infty,
$$
for every $\beta < \frac{1}{4e (\Delta_1 + \Delta_2)^2}$.
\end{theo}
\noindent {\bf Proof.} We follow a standard strategy to prove the exponential integrability $S^*_Z$ by estimating
the $p$-th moments of $S^*_Z$.   To estimate the $p$-th moments we shall not use
Lemma \ref{lemme} which yields  badly behaving constants $C_p$, as $p\to \infty$, due to the use of Burkholder's inequality. Hence, we shall use Theorem \ref{theo-gen} instead.

\medskip
First remark that $\|S_Z^*\|_p$ ($p\ge 2$) is bounded by $K_p$ times the sum of the two terms in (\ref{series1}) and (\ref{series2}).
Notice that $$
\|\DD_{\ell+k} Z_{\ell}\|_p \le \|\E^{\ell+k+1} Z_{\ell} -Z_\ell\|_p
 + \|\E^{\ell+k} Z_{\ell} -Z_\ell\|_p.
 $$
Then, using Minkowski's inequality in $\ell^{2}$ we obtain that for every $N\in \N$,
$$
\sum_{k=0}^\infty
\Big( \sum_{\ell=0}^\infty \|\DD_{\ell+k} Z_{\ell}\|_p^{2}
\Big)^{1/2}\le 2 \sum_{k=0}^\infty \Big( \sum_{\ell=0}^\infty \| Z_{\ell}-\E^{\ell+k}Z_\ell\|_\infty^{2}
\Big)^{1/2}=2 \Delta_1\, .
$$
On the other hand, using
 $$\| \mathcal{D}_n Z_{n+k}
\|_p \le 2 \|\mathbb{E}^{n+1} Z_{n+k}\|_p,
$$
 we have
$$
\sum_{k=1}^\infty
\Big( \sum_{n =0}^\infty \|\DD_n
Z_{n+k}\|_p^2\Big)^{1/2}
\le 2 \sum_{k=1}^\infty \Big( \sum_{n =0}^\infty \|\E^{n+1}
Z_{n+k}\|_\infty^{2}\Big)^{1/2} = 2 \Delta_2\, .
$$
Thus we have proved $\| S_{Z}^* \|_p \le 2 K_p \Delta$ with $\Delta=\Delta_1 + \Delta_2$ for $p\ge 2$.
Let $\beta>0$. We have
$$
\E({\rm e}^{\beta (S_Z^*)^2})= \sum_{p=0}^\infty
\frac{\beta^p \|S_{Z}^* \|_{2p}^{2p}}{p!}\le 1 +\sum_{p=1}^\infty
\frac{\beta^p(2K_{2p}\Delta)^{2p}}{p!}\, .
$$
Since $K_{2p}^{2p}\le \big(\frac{2p}{2p-1}\big)^{2p} p^p\sim {\rm e} p^p$ and, by
Stirling's formula, $p!\sim \big(\frac{p}{\rm e}\big)^p \sqrt{2\pi p}$,
we infer that $\E({\rm e}^{\beta (S_Z^*)^2})<\infty$ as soon as
$\beta < (4{\rm e}\Delta^2)^{-1}$. \hfill $\square$

\subsection{Series of the form $\sum_{n=0}^\infty a_nW_n$}

In our applications, we shall be concerned with the situation where
$$Z_n=a_nW_n
$$ with $(a_n)_{n\in \N}$ a deterministic sequence and
$(W_n)_{n\in \N}$ a stationary sequence or at least a sequence
behaving somehow closely to a stationary one. When $p\ge 2$
(resp. when $1<p<2$), we are going to
control the $L^p$-moment of $\sum a_n W_n$ by the $\ell^2$-moment
(resp. the $\ell^p$-moment)
of $(a_n)$. Before stating the result, let us state
Cauchy's condensation principle whose proof is easy.
\smallskip

\begin{lem}
	Let $(u_n)_{n\in \N}$ be positive numbers such that
	$u_{n+m}\le Ku_n$ for every $n,m\in \N$ and for some $K>0$. Then the series
	$\sum_{\ell\in \N}  u_\ell$ converges if and only if the series   $\sum_{\ell\in \N}
	2^{\ell} u_{2^\ell}$ converges.
\end{lem}

\begin{theo}\label{theo-Lp-weight}
Let $1<p\le \infty$. Let $(W_n)_{n\in \N}\subset L^p(\Omega,\A,\P)$
be such that
\begin{equation}\label{Lp-cond}
\widetilde \Delta_p:= \sum_{\ell=0}^\infty
\frac1{(\ell+1)^{1/p}} \Big(\sup_{k\in \N}
\|W_k-\E^{k+\ell-1}W_k\|_p+ \sup_{m\in \N}
\|\E^{m+1}W_{m+\ell}\|_p\Big)<\infty\, .
\end{equation}
Then the series $\sum_{n\in \N}
a_n W_n $ converges $\P$-a.e. for every $(a_n)_{n\in \N}\in \ell^{p'}$.
Moreover, when $1<p<\infty$,  there exists $C_p>0$
(independent from $(a_n)_{n\in \N}$ and $(W_n)_{n\in \N}$) such that \begin{equation}
\|\, \sup_{N\in \N}
|\sum_{n=0}^Na_nW_n|\, \|_p \le C_p \widetilde \Delta_p \|(a_n)_{n\in \N}\|_{\ell^{p'}}
\, ;
\end{equation}
 and when $p=\infty$,
we have $$
\E {\rm exp} \big(\beta \sup_{N\in \N}
|\sum_{n=0}^Na_nW_n|^2\big)<\infty
$$
for every $\beta<(4{\rm e}\widetilde \Delta_\infty \|(a_n)_{n\in \N}\|_{\ell^2})^{-1}$.
\end{theo}
\noindent {\bf Proof.}
Let $1<p<\infty$. By Theorem \ref{theo-gen} and Lemma \ref{lemme},  we only have to check that
\begin{equation}\label{Lp-cond-eq}
 \sum_{\ell=0}^\infty
2^{\ell(1-1/p)} \Big(\sup_{k\in \N}
\|W_k-\E^{k+2^{\ell}-1}W_k\|_p+ \sup_{m\in \N}
\|\E^{m+1}W_{m+2^\ell}\|_p\Big)<\infty\, .
\end{equation}
This is actually equivalent to the condition \eqref{Lp-cond},
by the Cauchy condensation principle. Let us check the condition in the Cauchy condensation principle.
Notice that
$$
\sup_{m\in \N}
\|\E^{m+1}W_{m+\ell+1}\|_p \le \sup_{m\in \N}
\|\E^{m+2}W_{m+\ell+1}\|_p \le \sup_{m\in \N}
\|\E^{m+1}W_{m+\ell}\|_p\, ,
$$
and that, by the Burkholder inequality \eqref{Burkholder0}, for every $k, \ell\in \N$ and every $m\ge 1$
\begin{eqnarray*}
C_p
\|W_k-\E^{k+\ell-1}W_k\|_p
& \ge &
\|\big (W_k-\E^{k+\ell+m-1}W_k)^2+ (\E^{k+\ell+m-1}W_k-\E^{k+\ell-1}W_k)^2
\big)^{1/2}\|_p\\
&\ge & \|W_k-\E^{k+\ell+m-1}W_k\|_p
\end{eqnarray*}

The case $p=\infty$ can be proved similarly basing on Theorem
\ref{theo-bounded}.
 \hfill $\square$

\medskip

To conclude this section we shall prove that, when $p\ge 2$, there are situations where the condition $(a_n)_{n\in \N}\in \ell^{2}$ in the above theorem
is also necessary for the $\P$-a.e. convergence of $\sum_{n\in \N}
a_n W_n $.

\smallskip

\begin{defn}\label{Riesz}
We say that a sequence $(W_n)_{n\in \N}\subset
L^2(\Omega,\A,\P)$ is a Riesz system if there exists $C>0$ such that
for every $(b_n)_{n\in \N}\in \ell^2$,
$$
C^{-1}\|(b_n)_{n\in \N}\|_{\ell^2}\le\|\sum_{n\in \N} b_nW_n\|_2\le C \|(b_n)_{n\in \N}\|_{\ell^2}\, .
$$
If moreover $(W_n)_{n\in \N}$ is complete in $L^2(\Omega,\A,\P)$, we
say that is a Riesz basis.
\end{defn}

\begin{theo}\label{NSC}
Let $2<p\le \infty$. Suppose that $(W_n)_{n\in \N}\subset L^p(\Omega,\A,\P)$ such that
\begin{equation}
 \sum_{\ell=0}^\infty
\frac1{(\ell+1)^{1/p}} \Big(\sup_{k\in \N}
\|W_k-\E^{k+\ell-1}W_k\|_p+ \sup_{m\in \N}
\|\E^{m+1}W_{m+\ell}\|_p\Big)<\infty\,
\end{equation}
and that $(W_n)_{n\in \N}$ is
is a Riesz sequence in $L^2(\Omega, \mathcal{A}, \mathbb{P})$.
Then the series $\sum_{n\in \N}
a_n W_n $ does not converge $\P$-a.e. for every $(a_n)_{n\in \N}$ such that
$\sum_{n\in \N} |a_n|^2=\infty$.
\end{theo}
\noindent {\bf Remark.} We do not know whether the series is $\P$-a.e.
divergent under the above conditions.

\smallskip

\noindent {\bf Proof.} We proceed as in the proof of Theorem 2.6 of
\cite{FanETDS}. Recall first the following Paley-Zygmund inequality
(see Kahane \cite{Kahane} for the case $q=2$, the proof being the same for general $q>1$):
Let $Z\in L^q(\Omega,\A,\P)$ be non negative ($q>1$). For any $0<\lambda <1$, we have
$$
\P(Z\ge \lambda \E Z)\ge \Big( (1-\lambda)\frac{\E Z}{\|Z\|_q}
\Big)^{q/(q-1)}\, .
$$

We apply the inequality  with $q=p/2$ and $Z=Z_N= \sup_{0\le n\le N} |\sum_{k=0}^n
a_k W_k|^2$. By Theorem \ref{theo-Lp-weight} and the hypothesis that $(W_n)_{n\in \N}$
is a Riesz sequence, we know that
there exist $C,D>0$ such that
$$
D\sqrt{\sum_{k=0}^N|a_k|^2} \le \E Z_N\le \|Z_N\|_q\le C \sqrt{\sum_{k=0}^N|a_k|^2}\, .
$$
Hence, we infer that
$$
\P\big(Z_N\ge \lambda D\sum_{k=0}^N|a_k|^2\big)\ge \Big((1-\lambda)\frac{D}{C}\Big)^{q/(q-1)}\, .
$$
Since, $\sum_{n\in \N}|a_n|^2=\infty$, the result follows. \hfill $\square$

\section{Analysis using decreasing filtration}\label{SectDecreasing}

We can also use decreasing filtrations to analyze our series. Recall that
$(\Omega, \A, \mathbb{P})$ is a probability space and $(Z_n)_{n\ge 0}\subset L^1(\Omega,\A,\P)$ is a
sequence of  random variables such that $\mathbb{E} Z_n =0$
for all $n$. Our object of study is  the
random series
\begin{equation}
\sum_{n=0}^\infty Z_n(\omega).
\end{equation}
Suppose that we are given an decreasing  filtration $(\B_n)_{n\in \N}$. Assume that
$\B_0=\A$.  Let
$\B_\infty:=\bigcap_{n=0}^\infty \B_n$. We will suppose that
\begin{equation}\label{H1}
\forall n, \quad \mathbb{E}(Z_n|\B_\infty)=0
\end{equation}
which implies $\mathbb{E} Z_n =0$.

For $n\ge 0$, denote the partial
sum $$
S_n = \sum_{k=0}^n Z_k.
$$
For $N\ge 0$, define the maximal function
$$
      S_N^*(\omega) =\max_{0\le n \le N} |S_n(\omega)|.
$$
The basic idea is to convert the random series into reverse martingales.

\smallskip
For every $n\in \N\cup\{\infty\}$, denote $\E^n:= \E(\cdot|\B_n)$ and
$$\mathfrak{d}_n =\E_{n}-\E_{n+1}.
$$

Let us state the useful properties of the operators $\mathfrak{d}_n$ in the following proposition.
The following lemma is the same as Lemma \ref{martingale}.

\begin{lem} \label{martingale-dec} Let $h \in L^1(\Omega, \A, \P)$ and $f, g\in L^2(\Omega, \A, \P)$. The operators $\mathfrak{d}_n$ have the following properties:\\
\indent {\rm (1)}\
For and $n\ge 0$, $\DD_n$ is $\B_{n}$-measurable and $\E_{n+1}\DD_{n} f=0$.\\
\indent {\rm (2)}\
For  any distinct integers $n$ and $m$,
$\mathfrak{d}^n f$ and $\mathfrak{d}^m g$ are orthogonal.\\
\indent {\rm (3)}\ For  any $N_1<N_2$ we have
$$
     \sum_{n=N_1}^{N_2} \|\mathfrak{d}_n f\|_2^2 = \|\E_{N_1} f - \E_{N_2 +1} f\|_2^2.
$$
\end{lem}


\medskip
The first assertion implies that for any sequence $(f_n)\in L^1(\Omega, \A, \P)$,
$(\mathfrak{d}_n f_n)$ is a sequence of reverse martingale differences.
\medskip

For any integral random variable such that $\mathbb{E}(Z|\B_\infty)=0$, we may decompose it into martingale as follows.
We have the following analogue of Lemma \ref{decomposition}, whose proof is the same (it does not matter here that we are dealing with reverse
martingale differences rather that martingale differences).

\begin{lem} \label{decomposition2} Let
$Z\in L^p(\Omega,\A,\P)$, $p\ge 1$ with $\E (Z|\B_\infty)=0$. We have the decomposition
\begin{equation}\label{decomp}
Z=\sum_{n=0}^\infty \DD_{n}Z \, ,
\end{equation}
where the series converges in $L^p$ and $\P$-a.s. Moreover we have
\begin{equation}\label{Rio}
\|Z\|_p\le \max(1,\sqrt{p-1}) \Big(\sum_{n=0}^\infty \|\mathfrak{d}_{n}Z\|_p^{p'}
\Big)^{1/p'}\, ,
\end{equation}
where $p':=\min(2,p)$ and, if $p>1$,
\begin{equation}\label{Burkholder}
\Big\|\Big(\sum_{n=0}^\infty |\mathfrak{d}_{n}Z|^2\Big)^{1/2}\Big\|_p \le C_p
\|Z\|_p\, .
\end{equation}
\end{lem}

\medskip

We call $\mathfrak{d}_n Z$ the {\em  $n$-th order detail} of $Z$ with respect to $(\B_n)$.
Since $\B_n$ is decreasing, we say that $\mathfrak{d}_n Z$ is a detail of
 higher order than $\mathfrak{d}_m Z$ when $n<m$.

\medskip

\begin{theo}\label{theo-gen-decreasing}
Let $(Z_n)_{n\in \N }\subset L^p(\Omega,\A,\P)$, $p>1$, be such that
 $\E_\infty (Z_n)=0$ for every $n\in \N$. Then,
for every $N\in \N  $,
\begin{equation}
\| S_N^* \|_p
\le 2K_p \Big( \sum_{k=1}
^N \Big( \sum_{\ell=k}^N \|\mathfrak{d}_{\ell-k}Z_\ell\|_{p'}^{p'}\Big)^{1/p'}+ \sum_{k=0}^\infty \Big(\sum_{\ell=0}^N
\|\mathfrak{d}_{\ell+k} Z_{\ell}\|_{p'}^{p'}\Big)^{1/p'}\Big)\, .
\end{equation}
Consequently, the series $\sum_{n\in \N}g_n $
converges $\P$-a.s. and in $L^p(\Omega,\A,\P)$ under the following conditions
\begin{equation}\label{series1+} \sum_{k=0}^\infty\Big( \sum_{n=0}^\infty \|\mathfrak{d}_{n+k}
Z_{n}\|_{p'}^{p'}\Big)^{1/p'}<\infty
\end{equation} and
\begin{equation}\label{series2+}
\sum_{k=1}^\infty\Big( \sum_{n =0}^\infty \|\mathfrak{d}_n
Z_{n+k}\|_{p'}^{p'}\Big)^{1/p'}<\infty\,.
\end{equation}
If $(Z_n)$ is adapted to $(\B_n)$, the condition (\ref{series2+}) on the higher order details is trivially satisfied.
\end{theo}

The proof being similar to that of Theorem \ref{theo-gen}, we leave it
to the reader. Some similar arguments can be found  \cite{FanETDS}.

\medskip

In order to apply Theorem \ref{theo-gen-decreasing}, it will be convenient to use the sufficient conditions in the next lemma.
\medskip

\begin{lem}\label{lemme-decreasing}
Let $(Z_n)_{n\in \N }\subset L^p(\Omega,\A,\P)$, $p>1$, with $\E_\infty Z_n =0$ for all $n$. The condition \eqref{series2+} on the higher order details
is satisfied  if
\begin{equation}\label{series1bis}
\sum_{\ell=0}^\infty 2^{\ell(1-1/p)}\Big( \sum_{n=2^{\ell}}^\infty \|Z_n - \E_{n-2^{\ell} +1}Z_n\|_{p'}^{p'} \Big)^{1/p'}<\infty\,.
\end{equation}
The condition \eqref{series1+} on the lower order details is satisfied
if
\begin{equation}\label{series2bis}
\sum_{\ell=0}^\infty 2^{\ell(1-1/p)}\Big( \sum_{n=0}^\infty
\|\E_{n +2^\ell -1} Z_n \|_{p'}^{p'} \Big)^{1/p'}<\infty\, .
\end{equation}
\end{lem}

The proof is similar to that of Lemme \ref{lemme} and we leave it to the reader.

Results similar to Theorems \ref{theo-bounded}, \ref{theo-Lp-weight} and \ref{NSC} holds.

\section{Convergence of ergodic series}\label{SectES}
Let $(X, \mathcal{B}, \mu, T)$ be a measure-preserving dynamical
system. By an ergodic series we mean a series of the form
\begin{equation}\label{ES}
   \sum_{n=0}^\infty a_n f_n(T^n x)
\end{equation}
where it is assumed that the $f_n$'s are integrable with $\mathbb{E} f_n=0$ and that $(a_n)$ is a sequence of numbers. The almost everywhere convergence of such series was  studied in \cite{FanETDS} where the martingale method was already used, and which is a  motivation of our present study.

In this case, the natural filtration that we can use to analyze the series is the one defined by
$$\B_n =T^{-n}\B.$$
 It is a decreasing filtration. Let $L$ be the transfer operator associated to the
 dynamical system which can be defined by
 $$
     \int f \cdot Lg d\mu = \int f\circ T \cdot g d\mu \qquad (\forall f\in L^\infty(\mu), \forall g \in L^1(\mu)).
 $$

\begin{theo} \label{main} Assume that $(X, \mathcal{B}, T, \mu)$ is an ergodic measure-preserving dynamical system.
      Let $(f_n)_{n\in \N} \subset L^p(\mu)$, $1<p\le \infty$. Suppose
      \\
      \ \indent {\rm (H1)}\ \  $\forall n\in \N$,  $
      \lim_{m\to \infty} \|\mathbb{E} (f_n|T^{-m}\mathcal{B})\|_p=0$; \\
      \ \indent {\rm (H2)} \ $
          \sum_{i=0}^\infty 2^{\ell(1-1/p)} \sup_{n\ge 0} \|L^{2^\ell} f_n\|_p <\infty
      $.\\
      Then for any complex sequence $(a_n)_{n\in \N}\subset \mathbb{C}$ such that $\sum_{n=0}^\infty |a_n|^{p'}<\infty$,
      the ergodic series $\sum_{n=0}^\infty a_n f_n(T^n x)$ converges a.e., and in $L^p(X, \mathcal{B}, T, \mu)$ if $p<\infty$. Moreover,
      when $1<p<\infty$, $\sup_{N\in \N}|\sum_{n=0}^N a_n f_n(T^n x)|\in L^p(X, \mathcal{B}, T, \mu)$ and when $p=\infty$, $\E(\exp(\beta \sup_{N\in \N}|\sum_{n=0}^Na_nf_n\circ T^n|^2))
      <\infty$ for some  $\beta>0$.
      \end{theo}

\noindent{\bf Proof.}  Since $(f_n\circ T^n)$ is adapted to $(\B_n)$, we can apply Theorem \ref{theo-gen-decreasing} by just checking the condition (\ref{series2bis}) on the higher order details.
 The checking is easy and is a direct consequence of the fact that for $m\ge n$ we have
 $$
      \mathbb{E} (f_n\circ T^n|T^{-m}\B)
      = (L^{m-n}f) \circ T^m.
 $$
 Then, letting $Z_n =a_n f\circ T^n$, it suffices to notice
 $$
    \|\E_{n+2^\ell - 1} Z_n \|_p^{p'} = |a_n|^{p'} \|(L^{2^\ell-1}f_n) \circ T^m\|_p^{p'}
    =|a_n|^{p'} \|L^{2^\ell-1}f_n\|_p^{p'}.
 $$

 $\Box$

 \medskip

When $p=\infty$, the condition $(H2)$ reads $\sum_{i=0}^\infty 2^{\ell} \sup_{n\ge 0} \|L^{2^\ell} f_n\|_\infty <\infty$.
Under this last condition,  the above theorem was proved in \cite{FanETDS}  and the
 exponential integrability of the maximal function was not mentioned in \cite{FanETDS} but it follows
 from the obtained estimates there. If one is only interested in almost everywhere convergence but not in the exponential integrability, Theorem \ref{main} relax the condition
 in \cite{FanETDS} by requiring only $L^p$-integrability, $2\le p\le \infty$ and weakening the exponent of $2^\ell$ in (H2).

 \medskip

 In the study of dynamical systems, the decay of $L^n f$ (measured by $L^\infty$-norm or the H\"{o}lder norm) was extensively studied for regular functions $f$ like H\"{o}lder functions. The above
 theorem shows that for the almost everywhere convergence of the ergodic series, weaker regularity
 would be sufficient and that there is an interest  to study the decay of  $L^n f$ measured by $L^2$-norm. Several situations where such a decay is 
 measured, for unbounded functionals, may be found for instance 
 in \cite{CM-reverse}, section 3.2 or in \cite{F1999}.

 \medskip

 To conclude the section, let us show  that the condition (H2) in  Theorem \ref{main} is sharp.
 We shall use the notation $L_0(x)=1$, $L_1(x)=\max(1,\log x)$ and
 $L_m(x)=L_1\circ \cdot \circ L_1(x)$ (where $L_1$ appears $m$ times).

  Let $T$ be the measure preserving transformation on $([0,1),\B([0,1)),
  \lambda)$ ($\lambda$ being the Lebesgue measure) defines by
  $Tx=2x\mod 1$. Then $\|L^ng\|_1\to 0$ for
  every $g\in L^1([0,1),\B([0,1)), \lambda)$ such that $\int g(x) d x=0$.

 \begin{prop}\label{example-gaposhkin-dyn}
 	Consider the dynamics $([0,1),\B([0,1)),
 	\lambda, T)$ where $\lambda$ is the Lebesque measure and
 	 $Tx=2x\mod 1$.
For every  $m\in \N$, there exist  $(a_n)_{n\in \N}
 \in \ell^2$ and $f\in L^p([0,1),\B([0,1)), \lambda)$  for every $1\le p<\infty$ with  $\int f(x) d x=0$,  such that $\|L^{2^n} f\|_2=O( \frac{2^{-n/2}}{L_m(2^n)})$ and that the series $\sum_{n\in \N} a_n f\circ T^n$
 diverges almost everywhere.
 \end{prop}

 The proposition follows from Proposition \ref{example-gaposhkin} in the next section and the fact that for $Tx =\{2x\}$ and for $f\in L^2([0,1),\B([0,1)), \lambda)$
 with $\int f(x) d(x) d x=0$, we have
 $\|L^nf\|_2=\omega_2(f, 2^{-n})$
 (Theorem 2 in \cite{F1999}), where $\omega_2$ is the $L^2$-modulus of continuity.

\section{Convergence of dilated series}\label{SectDS}

We want to apply our general results in Section \ref{SectIncreasing} to the study of the following series, called dilated series,
\begin{equation}\label{dilated}
\sum_{k=0}^\infty a_k f(n_kx),
\end{equation}
 where we assume $f\in L^p([0,1],\B,\lambda)$  for some $p>1$
 ($\lambda$ being the Lebesgue measure and $\B$ the Borel $\sigma$-algebra), $(a_k)_{k\in \N}\in\ell^{p'}(\N)$ ($p': =\min (2, p)$), and
$(n_k)_{k\in \N}$ an increasing sequence of positive  integers.

\medskip

There is an extensive literature on the topic for $p=2$, which is our main concern here.
Let us mention the  surveys \cite[(1966)]{Gaposhkin66} by Gaposhkin and
\cite[(2009)]{BW} by Berkes and Weber. For more recent results, one may refer to Weber \cite{Weber}, see also the references therein.

\subsection{Lacunary dilated series}

Before stating our next result we need a definition.

\begin{defn}\label{modulus}
For every $f\in L^p([0,1],\B,\lambda)$, we define its $L^p$-modulus of continuity $\omega_p(\cdot,f)$ by

$$
\omega_p(\delta,f):= \sup_{0\le h\le \delta} \|\tau_h f-f\|_p\qquad
\forall \delta\in [0,1]\,
$$
where $\tau_h f(x) := f(x+h)$.
\end{defn}

\begin{theo}\label{theo-dilated}
Let $f\in L^p([0,1],\B,\lambda)$, $1< p\le \infty$, be such that
\begin{equation}\label{condition*}
\sum_{n\in \N} \frac{\omega_p(2^{-n},f)}{n^{1/p}}<\infty.
\end{equation}
Let $(n_k)_{k\in \N}$ be a Hadamard lacunary sequence of positive integers. Then
   for every $(a_k)_{k\in \N}\in \ell^{p'}(\N)$, the series
$\sum_k a_kf(n_k\cdot)$ converges $\lambda$-a.s. and in $L^p$.
Moreover, if $1<p<\infty$, we have $$
\sup_{N\in \N}|\sum_k ^Na_kf(n_k\cdot)|\in L^p([0,1],\B,\lambda);$$ and if $p=\infty$, there exists $\gamma>0$ such that
\begin{equation}\label{KW}
\E\exp\Big(\gamma \sup_{N\in \N}|\sum_k ^Na_kf(n_k\cdot)|^2\Big)<
\infty.
\end{equation}
\end{theo}
\noindent {\bf Remarks.} If $f$ is a trigonometric polynomial, the condition \eqref{condition*} holds with $p=\infty$. In this case
Kuelbs-Woyczy\'nski \cite[Corollary 4.1]{KW} proved (\ref{KW})   for all $\gamma>0$.

\medskip

Let us mention some previous results closely related to Theorem
\ref{theo-dilated} with $p=2$. Gaposhkin \cite{Gaposhkin68}
(see also Theorem 2.4.2 in \cite{Gaposhkin66} and the footnote there), proved that if $\label{condition}\sum_{n\in \N} \omega_2(2^{-n},f)
<\infty$,  then the series
$\sum_{k\in \N}a_{k}f(n_k x)$ converges almost everywhere for every $(a_n)_{n\in \N}\in \ell^2$. Here the lacunarity of $(n_k)$ is not needed. But under the assumption of lacunarity,
we gain a factor $\sqrt{n}$ in (\ref{condition*}).  Moreover,
in \cite{Gaposhkin67}, Gaposhkin proved that if $f$ is defined by a lacunary Fourier series with $\omega_2(f,2^{-n})=O(n^{-\gamma})$ for some $\gamma>1/2$, the same conclusion holds, and the condition $\gamma >\frac{1}{2}$
is
sharp by Proposition \ref{example-gaposhkin}, quoted from Gaposhkin
\cite{Gaposhkin67}. Proposition \ref{example-gaposhkin}  also proves the sharpness of (\ref{condition*}) in Theorem \ref{theo-dilated}.

 To conclude our discussion let us mention that our proof makes use of dyadic martingales while the proof of Gaposhkin is based on Lebesgue's
differentiation theorem. It is well-known that these two objects are
linked (see for instance \cite{JKRW}). In our context such a link
appears in the proof of Lemma \ref{lemme-dyadic}, see the remark after it.

\smallskip

\begin{prop}{\rm (Gaposhkin} \cite[Theorem 3]{Gaposhkin67}\rm{)}\label{example-gaposhkin}
For every $m\in \N$, there exists $f$ which is in $L^p([0,1],\B,\lambda)$
for every  $1\le p<\infty$, such that $\omega_2(f,2^{-n})= O(\frac1{\sqrt n L_m(n)})$ and
a sequence $(a_n)_{n\in \N}\in \ell^2$ such that the series
$\sum_{n\in \N}a_nf(2^nx)$ diverges almost everywhere.
\end{prop}
\noindent{\bf Proof. } Fix $m\in \N$.
Define
$$f(x):= \sum_{k= 1}^\infty \frac{\sin(2^k \cdot 2\pi x)}{k\prod_{i=0}^m
L_i(k)}; \quad a_n:= \frac1{\sqrt{n \prod_{i=0}^{m-1}
L_i(n)}L_m(n)}  \ (\forall n\ge 1).
$$
 Since $f$ is lacunary and in $L^2$, it is in every $L^p$
by Theorem 8.20, Chap V, vol I (page 215) of \cite{Zygmund}.
Then, the rest of the proof is exactly as in Gaposhkin
\cite{Gaposhkin67}. \hfill $\square$

\medskip

Now we show that the condition $(a_n)_{n\in \N}\in \ell^2$ in Theorem \ref{theo-dilated}  may be necessary.

\begin{theo}\label{theo-dilated2}
Let $f\in L^p([0,1],\B,\lambda)$, $2< p\le \infty$, satisfying \eqref{condition*}. Let $(n_k)_{k\in \N}$
  be a Hadamard-lacunary sequence of positive integers. Suppose  that
  $(f(n_k\cdot ))_{k\in \N}$ is a Riesz system. Then, for every sequence
  $(a_n)_{n\in \N}$, with $\sum_{n\in \N}|a_n|^2=\infty$, the series
  $\sum_{k\in \N}a_k f(n_k\cdot )$ is not $a.e.$ convergent.
\end{theo}


\subsection{Proofs of Theorem \ref{theo-dilated} and Theorem \ref{theo-dilated2}}

As for several results on the topic (see e.g. Theorem 2.1 of \cite{BW}),
dyadic martingales were used. We also use the dyadic filtration.

\medskip

 For every $n\in \N$, denote by $\F_n$ the $\sigma$-algebra generated the family $\mathfrak{I}_n$ of
the intervals $I_{n, k}:=[\frac{k}{2^n},\frac{k+1}{2^n}]$, $0\le k\le 2^n-1$. We will choose our filtration
$(\A_k)$ to be $(\F_{m_k})$ for some suitable increasing sequence of integers $(m_k)$.


For every increasing sequence $(m_k)_{k\in \N}$, we do have $\bigvee_{k\in \N}
\F_{m_k}=\B$.  For every $f\in L^p([0,1],\B,\lambda)$, $p\ge 1$,
$(\E(f|\F_{m_k})_{k\in \N}$ converges  to $f$ $\lambda$-a.s. and in
$L^p([0,1],\B,\lambda)$.

\smallskip

In the next lemma, which would be known to specialists, we control the rate
of approximation  of $f\in L^p$ by $\E(f |\F_n)$, by using the $L^p$-modulus of continuity
of $f$.

\smallskip We convention that our functions are periodically extended to the whole line
$\mathbb{R}$.

\begin{lem}\label{lemme-dyadic}
Let $f\in L^p([0, 1], \B, \lambda)$, $1\le p\le \infty$, and $n\in \N$. We have
\begin{equation}\label{control-dyadic}
\|f-\E(f|\F_n)\|_p\le 2 \Omega_p(2^{-n},f)\, .
\end{equation}
\end{lem}
\noindent {\bf Remark. } When, $p=2$, the lemma  improves  Lemma 2.1 in \cite{BW}. Notice that in the proof we make use of $\sum_{k=0}^{2^n-1}  m_{I_{n, k}} (\tau_{x-2^{-n}k} f) {\bf 1}_{I_{n,k}}(x)
={2^n}\int_x^{x+\frac1{2^n}}f(u)du$, which is related to Lebesgue's differentiation theorem.

\noindent {\bf Proof.} We give the proof when $p<\infty$,
the case $p=\infty$ being obvious. For simplicity, for any interval $I$ we write
$$
   m_I(f) = \frac{1}{|I|}\int_I f(x) d x
$$
where $|I|$ denotes the length of $I$.
Let $n\in \N$. We have, for $\lambda$-a.e. $x\in [0,1]$,
$$
\E(f|\F_n)(x)= \sum_{k=0}^{2^n-1} m_{I_{n, k}}(f) {\bf 1}_{I_{n, k}}
(x)\, .
$$
Since $f(x)=\sum_{k=0}^{2^n-1} f(x){\bf 1}_{I_{n, k}}(x)$, we can  write
\begin{align*}
f(x)-\E(f|\F_n)(x)= \varphi_n(x)+\psi_n(x)
\end{align*}
where
\begin{align*}
\varphi_n(x):=\sum_{k=0}^{2^n-1} [f(x) - m_{I_{n, k}} (\tau_{x-2^{-n}k} f)] {\bf 1}_{I_{n,k}}(x)
\\
\psi_n(x):= \sum_{k=0}^{2^n-1} [ m_{I_{n, k}} (\tau_{x-2^{-n}k} f - m_{I_{n, k}
} ( f))] {\bf 1}_{I_{n,k}}(x).
\end{align*}
Using that
$\frac{f(x)}{2^n}=f(x) \int_{\frac{k}{2^n}}
^{\frac{k+1}{2^n}}du$ and using the H\"older inequality in the following integral  with respect to $du$, we get that
\begin{eqnarray*}
\|\varphi_n\|_p^p &=& \sum_{k=0}^{2^n-1} \int_{I_{n, k}} |f(x)- \tau_{x-2^{-n} k} f(u)|^p du\\
&=&
\sum_{k=0}^{2^n-1}
 \int_{\frac{k}{2^n}}
^{\frac{k+1}{2^n}} \Big|  2^n \int_{\frac{k}{2^n}}
^{\frac{k+1}{2^n}} \big({f(x)}-f(u+x-k2^{-n})\big)\, du\Big|^pdx\\
&\le& 2^{n}\sum_{k=0}^{2^n-1}
 \int_{\frac{k}{2^n}}
^{\frac{k+1}{2^n}} \Big[  \int_{\frac{k}{2^n}}
^{\frac{k+1}{2^n}}\big|{f(x)}-f(u+x-k2^{-n})\big|^p\, du\, \Big] dx
\\
&=& 2^{np}\sum_{k=0}^{2^n-1}
 \int_{\frac{k}{2^n}}^{\frac{k+1}{2^n}}
    \Big( 2^{n(1-p)} \int_{0}^{\frac{1}{2^n}}\big|f(x) - f(x+v) \big|^p\, dv\, \Big) dx\\
& =& 2^n \int_0^{2^{-n}}dv  \int_0^1 |f(x) - f(x+v)|^p d x\\
& \le & \Omega^p_p(2^{-n},f).
\end{eqnarray*}
By a similar computation, we infer that
$$
\|\psi\|_p^p\le \Omega^p_p(2^{-n},f)\, ,
$$
and the result follows. \hfill $\square$

\medskip

We will also need the following lemma.

\begin{lem}\label{contraction}
Let $f\in L^p([0,1],\B,\lambda)$, $1\le p\le \infty$, with $\int_0^1f(u)du=0$. For any integers $n \in \N$ and $m\ge 1$,
we have
\begin{equation}\label{control-dilated}
\|\E(f(m\cdot) |\F_n)\|_p\le \frac{2^{n}}{m}\|f\|_p.
\end{equation}
\end{lem}
\noindent {\bf Proof.} Assume that $p<\infty$. Since $\int_0^1f(u)du=0$, the integral of $f$ over any interval of integral length is equal to zero.
So that
for any $a<b$, we have
$$
    \int_a^b f(x) dx = \int_a^{a+[b-a]} f(x) dx + \int_{a+[b-a]}^{a + [b-a]+\{b-a\}} f(x) dx
    = \int_a^{a + \{b-a\}}f(x) dx
$$
where $[x]$ and $\{x\}$ denote respectively the integral part and the fractional part of a real
number $x$.
It follows that
$$
  \Big|\int_{\frac{km}{2^n}}^{\frac{(k+1)m}{2^n}}
f(u) \, du\Big|^p \le \Big(\int_{0}^1
|f(u)| \, du\Big)^p \le \|f\|_p^p.
$$
Hence,
\begin{gather*}
\|\E(f(m\cdot) |\F_n)\|_p^p =\frac{2^{n(p-1)}}{m^p} \sum_{k=0}^{2^n-1}
\Big|\int_{\frac{km}{2^n}}^{\frac{(k+1)m}{2^n}}
f(u) \, du\Big|^p\le \frac{2^{np}}{m^p}\|f\|_p^p\, .
\end{gather*}
The proof when $p=\infty$ follows similarly.
\hfill $\square$

\medskip
\noindent {\bf Remark.} Notice that if $m\equiv \ell\, \, {\rm mod}
\,\,   2^n$, with $0\le \ell \le 2^n-1$, we actually have $\|\E(f(m\cdot) |\F_n)\|_2\le \frac{\sqrt{\ell 2^{n}}}{m}\|f\|_2$. We can replace the norm $\|\cdot \|_2$ by  the norm $\|\cdot \|_1$.

\medskip

\noindent {\bf Proof of Theorem \ref{theo-dilated}.} It is easy to see that there exists an integer
$r\ge 1$, such that in each interval $[2^\ell, 2^{\ell+1}-1]$, there
are at most $r$ terms from the sequence $(n_k)_{k\in \N}$. Splitting our series into $r$ series we may and do assume that $r=1$.

\smallskip

For every $k\in \N$, define $m_k:=[\log_2 n_k]$, in other words
$2^{m_k}\le n_k <2^{m_k +1}$. The sequence  $(m_k)_{k\in \N}$ is strictly increasing and tends to the infinity. We are going to  apply Theorem
\ref{theo-gen} to
$$Z_k=a_k f(n_k\cdot), \quad \A_k=\F_{m_k}.
$$ By
Lemma \ref{lemme}, it suffices to check that
\begin{equation}\label{series1ter}
\sum_{\ell=0}^\infty 2^{\ell(1-1/p)}\Big( \sum_{k=0}^\infty a_k^{p'}
\|f(n_k\cdot) -\E(f(n_k\cdot |\F_{m_{k+2^l}})\|_p^{p'}\Big)^{1/p'}<\infty\, ,
\end{equation}
and that
\begin{equation}\label{series2ter}
\sum_{\ell=0}^\infty 2^{\ell(1-1/p)}\Big( \sum_{k=0}^\infty
a_k^{p'} \|\E(f_{n_k}\cdot )|\F_{m_{k-2^\ell}})\|_p^{p'}\Big)^{1/p'}<\infty\, .
\end{equation}
Using \eqref{control-dyadic}, we see that \eqref{series1ter} holds
as soon as
$$
\sum_{\ell=0}^\infty 2^{\ell(1-1/p)}\Big( \sum_{k=0}^\infty a_k^{p'}
\omega_p^{p'}(\frac{n_k}{2^{m_{k+2^\ell}}},f)\Big)^{1/p'}<\infty\,.
$$
Notice that as function of $\delta$, $\omega_2(\delta,f)$ is non-decreasing. Then
\eqref{series1ter} holds  as soon as
$$
\sum_{\ell=0}^\infty 2^{\ell(1-1/p)}\Big(
\sup_{k\in \N}\omega_p^{p'}(\frac{2n_k}{n_{k+2^\ell}},f)\Big)^{1/p'}
\le \sum_{\ell=0}^\infty 2^{\ell(1-1/p)}\Big(
\omega_p^{p'}(\frac{2}{q^{2^\ell}},f)\Big)^{1/p'}<\infty\,,
$$
which is equivalent to \eqref{condition*}, using once more the monotony of $\omega_p(\cdot, f)$.

It remains to prove \eqref{series2ter}. Using \eqref{control-dilated},
we see that \eqref{series2ter} holds as soon as

\begin{equation*}
\sum_{\ell=0}^\infty 2^{\ell(1-1/p)}\Big( \sum_{k=2^\ell}^\infty
a_k^{p'}\frac{2^{p'm_{k+1-2^\ell}}}{n_k^{p'}}\Big)^{1/p'}<\infty\, .
\end{equation*}
To conclude, just notice that
$$
\frac{2^{p'm_{k+1-2^\ell}}}{n_k^{p'}}\le 2^{p'} \frac{n_{k+1-2^\ell}^{p'}}{n_k^{p'}}
\le \frac{2^{p'}}{q^{p'2^\ell}}\, .
$$
\hfill $\square$

\subsection{Lacunary Davenport series}

As an example we apply Theorem \ref{theo-dilated} and Theorem \ref{theo-dilated2} to Davenport series.
For $\lambda>0$ï¼?we consider the function
$$
f_\lambda(x)= \sum_{m\ge 1}\frac{\sin(2 \pi mx)}{m^\lambda}\qquad
\forall x\in [0,1]\, .
$$
It is everywhere defined and continuous except at $x=0$ and belongs to $L^2([0, 1])$. It is H\"{o}lder continuous when $\lambda >1$.
When $\lambda=1$, we have $f_1(x)=-\{x\}$ (where $\{x\}= x -[x]-\frac{1}{2}$), hence $\omega(\delta , f_1)=O(\delta^{1/p})$, pour tout
$p>1$.

\smallskip

The study of the regularity of $f_\lambda$ for $0<\lambda <1$ needs more care (notice that when $1/2\le \lambda <1$ one may use the Hausdorf-Young theorem).

\smallskip

Let $0<\lambda <1$. By formula (2.3) p. 186 and formula
(1.18) p. 77 of Zygmund \cite[Vol. I]{Zygmund},
there exists $(c_n)_{n\ge 1}$ with $|c_n|\le C_\lambda n^{-1-\lambda}$
for some $C_\lambda >0$ and $\kappa_\lambda\in \R$, such that for every $x\in [0,1)$,
$$
f_\lambda(x)=\kappa_\lambda {\rm Im} \big((1-{\rm e}^{2i\pi x})^{\lambda-1}
\big)+\sum_{n\ge 1} c_n \sin(2\pi n x):=
g_\lambda(x)+h_\lambda(x)\, .
$$
It is not difficult to see that $h_\lambda$ is H\"older continuous
and that
\begin{equation}
     \omega_p(\delta, g_\lambda) =O(\delta^{\lambda -\frac{p-1}{p}})\, ,
\end{equation}
for every $p<(1-\lambda)^{-1}$.

\smallskip

These regularity properties allow us to apply Theorem \ref{theo-dilated}.
We shall see right now that when $\lambda>\frac{1}{2}$, it is also possible to apply
Theorem \ref{theo-dilated2}.
In fact,
for any $\lambda >1/2$, Wintner proved that $\{f_\lambda (n x)\}$
is complete in $L^2([0, 1])$. When $\lambda >1$, Hedenmalm, Lindqvist and Seip proved that   $\{f_\lambda (n x)\}$ is a Riesz basis, a fortiori $\{f_\lambda(n_kx)\}_{n_k\ge 1}$
is a Riesz sequence for any sequence $\{n_k\}$. When $1/2<\lambda
\le 1$, Br\'emont proved that $\{f_\lambda(n_kx)\}_{n_k\ge 1}$
 is a Riesz sequence for any Hadamard lacunary sequence $\{n_k\}$.

\begin{theo} \label{Thm-davenport}
Let $\lambda >1/2$. Suppose that $\{n_k\}\subset \mathbb{N}$ is lacunary in the sense of Hadamard.
	Then for every sequence $(a_k)_{k\in \N}$ the following are equivalent
	\begin{itemize}
	\item [$(i)$] The  series $\sum_{k=1}^\infty a_k f_\lambda(n_k x)$
	converges almost everywhere;
	\item [$(ii)$]  $\sum_{k=1}^\infty |a_k|^2 <\infty$.
	\end{itemize}
Moreover, if any of the above holds then, when $1/2<\lambda \le 1$, for every $p\le \frac{1}{1-\lambda}$, $\sup_{n\in \N} |	
\sum_{k=1}^N a_k f_\lambda(n_k x)|\in L^p([0,1))$; and; when $\lambda>1$, there exists $\beta>0$ such that $\int_0^1 {\rm e}^{\beta\sup_{n\in \N} |	
\sum_{k=1}^N a_k f_\lambda(n_k x)|^2}\, dx<\infty$.
\end{theo}
\noindent {\bf Remark.} The fact that $(ii)\Rightarrow (i)$ follows
from Gaposhkin \cite{Gaposhkin67}, but we provide here integrability properties of the maximal functions. The Theorem says that when
$\sum_{k=1}^\infty |a_k|^2 =\infty$ then the series $\sum_{k=1}^\infty a_k f_\lambda(n_k x)$
does not converge almost everywhere. Hence, one may wonder whether we have almost everywhere divergence. For a large class of sequences  $(a_k)_{k\in \N}$ and $(n_k)_{k\in \N}$ a positive answer follows from
Theorem 2.4.14 of \cite{Gaposhkin68}. Finally, notice that when $1/2< \lambda \le 1$,  $\sum |a_n|^2<\infty$ is not sufficient for
$\sum a_n f_\lambda(n x)$ to converge almost everywhere.
But when $\lambda >1$,  $\sum |a_n|^2<\infty$ is sufficient, as a consequence of Carleson theorem.

\begin{prop}
Let $0<\lambda \le 1/2$ and $p<(1-\lambda)^{-1}$. For every Hadamard lacunary sequence
$(n_k)_{k\in \N}$ and every $(a_k)_{k\in \N}\in \ell^p$
 the series $\sum_{k\in \N}
a_k f_\lambda(n_k x)$ converges almost everywhere. Moreover,
$\sup_{n\in \N} |	
\sum_{k=1}^N a_k f_\lambda(n_k x)|\in L^p([0,1)$.
\end{prop}



\section{Convergence of lacunary series with respect to Riesz products}
\label{SectRP}

The classical Riesz products in harmonic analysis are defined as follows
\cite{Zygmund}. Let
$(\lambda_n)_{n\in \N}$ be a sequence of positive integers such that $\lambda_{n+1}\ge 3 \lambda_n$ for all $n\ge 0$ and $(c_n)_{n\in \N}$ be a sequence of complex numbers such that $|c_n|\le 1$ for all $n\ge 0$. 
Then we can define a Borel probability measure  on  $ \mathbb{T}=\mathbb{R}/\mathbb{Z}$, denoted by
\begin{equation}\label{RandomRiesz2}
	\mu_{c} =\prod_{n=0}^\infty (1 + \mbox{\rm Re} \ c_n
	e^{2\pi i \lambda_n t}).
\end{equation}
Actually the partial products of the above infinite product are positive
trigonometric polynomials which converge to a measure $\mu_{c}$ in the weak-* topology of $C(\mathbb{T})^*$.  
Suppose that we are given a sequence of Borel functions $f_n$ on $\mathbb{T}$, say bounded.  A problem associated to Riesz products is the
$\mu_c$-a.e. convergence of the following lacunary series
\begin{equation}\label{RieszS}
     \sum_{n=0}^{\infty} a_n \big( f_n(\lambda_n x) -  \mathbb{E}_{\mu_c} f_n(\lambda_n \cdot)
     \big) . 
\end{equation}

It was independently proved in \cite{Fan1993} and \cite{Peyriere1990} that when $f_n(x) =e^{2\pi i x}$, the series (\ref{RieszS}) converges $\mu_c$-a.e.
if and only if $\sum_{n=0}^\infty |a_n|^2<\infty$.   For more general functions $f_n$, such results are not known. But under the conditions that $\lambda_n$ divises $\lambda_{n+1}$ for all $n$ and that $f_n$ are analytic or more precisely 
$$
   \exists \rho \in (0, 1), \quad \sup_{j\in \mathbb{Z
    		}} \rho^{-|j|} \sup_{n\ge 0} |\widehat{f}_n (j)|<\infty,
$$
Peyriere \cite{Peyriere1975} proved that  $\sum_{n=1}^\infty |a_n|^2<\infty$ implies 
the $\mu_c$-a.e. convergence of the  series  (\ref{RieszS}).
By using Theorem \ref{theo-gen-decreasing}, we can improve Peyri\`ere's result as follows

\begin{theo}\label{theo-riesz}
Assume that $\sup_{n\in \N}|c_n|<1$. Let $(f_n)_{n\in \N}$ be functions
on $[0,1]$, such that there exists $C>0$ and $\epsilon >0$ such that
$$
     \sup_{n\ge 0}\omega (t, f_n) \le \frac{C}{|\log t|^{1/2+\epsilon}}\, ,
$$
Then, for every $(a_n)_{n\in \N}\subset \ell^2$, 
the series  $\sum_{n=0}^{\infty} a_n \big( f_n(\lambda_n x) -  \mathbb{E}_{\mu_c} f_n(\lambda_n \cdot)$ converges $\mu_c$-a.e.
\end{theo}
We emphasize that we keep the divisibility condition on $\{\lambda_n\}$. Otherwise, more efforts are needed and for the moment we don't succeed.

\smallskip

We can do little more than Riesz products. 
Actually the above result holds not only for Riesz products but also  for non-homogeneous equilibrium states studied in \cite{FP}.  We shall 
now consider this situation. The proof of the Theorem will be given 
at the end of the paper.

\medskip

Let us recall the definition taken from \cite{FP}.
Let $\{S_n\}_{n\ge 1}$ be a sequence of finite sets of discrete topology. Assume that $\ell_n:= \mbox{\rm card } S_n \ge 2$ for all
$n \ge 1$. Consider the infinite product space $X :=
\prod_{n=1}^{\infty} S_n$ equipped with the product
topology. A compatible metric on $X$ may be defined as 
$$
d(x; y) = \frac{1}{\ell_1\ell_2\cdots \ell_n}
$$
where $ n =n(x, y) = \sup\{j \ge 1: x_i=y_i, \forall i=1,2 \cdots, j\}$ (with convention $\sup \emptyset = 0$). Let
$A= \{A_n\}_{n\ge 1}$ be a sequence of matrices such that 
$A_n \in M_{S_N, S_{n+1}}$,  meaning that the
rows of $A$ are indexed by $S_n$ and the columns by $S_{n+1}$. 
Suppose the entries of $A_n$ are
$0$ or $1$. Such a matrix is called an incidence matrix. We define a subspace $X_A$ of $X$ by
$$
X_A = \{x = (x_n) \in X: \forall n\ge 1, \ A_n(x_n, x_{n+1}) = 1\}.
$$
We call $X_A$ a non-homogeneous symbolic space. 
We always suppose that there exists an integer $M \ge  0$ such that
$$
\forall n\ge 1, \ \ \ \prod_{j=n}^{n+M}A_j >0
$$
($A > 0$ means that the entries of $A$ are all strictly  positive).
In this case, $X_A$ is said to be
transitive. If  all entries of every $A_n$ are equal to 1, 
we have  $X_A = X$.  We call $X$ the
full symbolic space.

A sequence $G=\{g_n\}_{n\ge 1}$ of non-negative functions defined on $X_A$ is called a
sequence of potentials if for any $\ge 1$,  $g_n(x)$ does not depend on the first $n-1$
coordinates of $x$ (so, we sometimes write $g_n(x) = g_n(x_n, x_{n+1}, \cdots))$.  It is said
to be normalized if for any $n \ge 1$, 
$$
\sum_{ y_n: A_n(y_n, x_{n+1})=1}
   g_n(y_n, x_{n+1}, \cdots) = 1 \quad (\forall x = (x_n) \in X_A).
$$
For $n \ge 1$,  let
$$ 
G_n(x) = \prod_{j=1}^n  g_n(x).
$$
Then define a sequence of averaging operators $P_n: C(X_A) \to C(X_A)$, where $C(X_A)$ is
the space of all continuous functions on $X_A$, by
$$
P_nf(x) =
\sum_{y_1, \cdots, y_n} G_n(y_1, \cdots, y_n, x_{n+1}, \cdots) 
f(y_1, \cdots, y_n, x_{n+1}, \cdots)
$$
where the sum is taken over all sequences $(y_1, \cdots,y_n)$ such that
$$
A_1(y_1, y_2)= \cdots = A_{n-1}(y_{n-1}, y_n)=A_n(y_n, x_{n+1}) = 1.
$$
It is easy to check that  $P_n$ is positive and $P_n1 = 1$. Therefore,  the adjoint operator
$P_n^*: M^+(X_A) \to M^+(X_A)$
admits fixed points, where $M^+(X_A)$
 is the space of all Borel
probability measures on $X_A$. A measure  $\mu \in M^+(X_A)$ is called a {\em (non-homogeneous)
equilibrium state} associated to $G = \{g_n\}$ if 
$P_n^* \mu = \mu$ for all $n\ge 1$.

\begin{theo}\label{gen-theo-FP}
[\cite{FP}] Let $G = \{g_n\}$ be a normalized sequence of potentials defined  on a
transitive symbolic space $X_A$. \\
\indent {\rm (a)}  The set of all equilibrium states associated to $G$ is a non-empty convex compact
set.\\
\indent {\rm (b)}   There is a unique equilibrium state if and only if for any $f \in  C(X_A)$, $P_nf(x)$
converges uniformly (in $x$) to a constant as $n \to \infty $.\\
\indent {\rm (c)}  There is a unique equilibrium state if
\begin{equation}\label{cond-eq}
   \inf_{n\ge 1} \inf_{x \in X_A} g_n(x) >0;
   \quad
\sup\left\{ \frac{G_n(x)}{G_n(y)}: x_1=y_1, \cdots, x_n=y_n \right\}
<\infty.
\end{equation}
\indent {\rm (c)} Under the condition in (c), there exist constants 
$D_1$ and $D_2$ such that
$$
   D_1 G_n(x) \le \mu(I_n(x))\le D_2 G_n(x)
$$
for all $x \in  X_A$ and all $n \ge  1$, where 
$I_n(x) = \{y\in X_A:  y_j = x_j \forall 1 \le j \le n\}$.
\end{theo}

For a function $f$ defined on $X_A$ and for $n\ge 1$, we define the
$n$-th variation of $f$ by
$$
 \mbox{\rm var}_n(f) =\sup\{|f(x)- f(y)|: x_1=y_1, \cdots, x_n =y_n\}. 
$$

A careful inspection of the proof of Theorem 4 of \cite{FP} gives the 
next theorem, which corresponds essentially to the case where 
$\alpha =1+\epsilon$ for $\epsilon>0$.  
\begin{theo}
	Let $\{g_n\}$ be a normalized sequence of potentials defined on a transitive symbolic space $X_A$.  Suppose there are constants $A > 0$ and $\alpha > 0$ such that for every $m > n > 1$,
\begin{equation}\label{cond-gn}
\mbox{\rm var}_m(\log g_n) \le \frac{A}{(m-n)^{\alpha}}\, .
\end{equation}
Then there is a unique equilibrium state $\mu$. 

\smallskip

Let $\{f_n\}$ be a sequence of functions such that $f_n$ depend only upon $x_{n+1}, x_{n+2}, \cdots$. Assume moreover that  there exists $B>0$ such that for every $m>n>1$,
$$
\|f_n\|_\infty \le B; \quad \mbox{\rm var}_m(f_n) \le \frac{B}{(m-n)^{\alpha}}.
$$
Then there exists $C>0$ such that for every $m>n>1$,
\begin{equation}\label{est-Pn}
\|P_mf_n\|_\infty \le C\frac{(\log (1+m-n))^{1+\alpha}}{(m-n)^{\alpha}}
\, .
\end{equation}
In particular, if $\alpha>1/2$, the series
\begin{equation}\label{SeriesG}
      \sum_{n=1}^\infty a_n \left( f_n(x) - \int f_n d\mu\right)
\end{equation}
converges $\mu$-a.e. if $\sum_{n=1}^\infty |a_n|^2<\infty$.
\end{theo} 

\begin{proof} 
	The condition \eqref{cond-gn}  $\{g_n\}$ ensures that condition 
	\eqref{cond-eq} holds, hence  the uniqueness of the equilibrium state, by the Theorem \ref{gen-theo-FP}. 
	
	\smallskip
	
	The fact that \eqref{est-Pn} holds may be proved as in the proof of Theorem 4 of \cite{FP}, see pages 111-112 there.
	
	\smallskip
	
Assume that $\alpha>1/2$.	Let $\mathcal{B}_n$ be the $\sigma$-field generated by the
	coordinate functions $x \mapsto x_j$ ($j=n+1, n+2, \cdots$). Then
	our series (\ref{SeriesG}) is adapted to the decreasing filtration
	$\{\mathcal{B}_n\}$. Then we will apply Theorem \ref{theo-gen-decreasing}. By Lemma \ref{lemme-decreasing},  what we have to check is just the condition (\ref{series1bis}), because  the condition (\ref{series2bis})
	is trivially satisfied by adapted series.
	
	Notice that $\mathbb{E}(f|\mathcal{B}_n) = P_nf$.  So,  
	taking $Z_n = a_n f_n$, we infer that for every $n,\ell\ge 0$,
	$$
	   \|E_{n+2^{\ell}-1}
	   Z_{n}\|_2 \le  |a_n|  (\|P_{n+2^{\ell}-1} f_n\|_2 
	   \le C\frac{|a_n| \ell^{1+\alpha}}{ 2^{\alpha \ell }}
	$$	
	for some $C>0$.
	Then
	$$
	\sum_{\ell=0}^\infty2^{\ell/2} \Big( \sum_{n=0}^\infty \|\E_{n+2^\ell-1}
	Z_{n}\|_2^2\Big)^{1/2}
	\le C
		\sum_{\ell=0}^\infty \frac{\ell^{1+\alpha}}{ 2^{\ell (\alpha-1/2)}} \Big( \sum_{n=0}^\infty |a_n|^2\Big)^{1/2} <\infty.
	$$
	Thus the condition (\ref{series1+}) is verified. 
\end{proof}


\medskip

\noindent {\bf Proof of Theorem \ref{theo-riesz}.} 
We can identify the full symbolic space $X$ with the circle
$\mathbb{T}$ by the map from $X$ to $\mathbb{T}$:
$$   
(x_n) \mapsto \sum_{n=1}^\infty \frac{x_n}{\ell_1 \cdots \ell_n}.
$$
Then the Riesz product is nothing but the equilibrium state associated to
$$
   \forall n\ge 0, \ \ \ 
    g_{n+1}(x) =\ell_{n+1}^{-1} (1 +  \mbox{\rm Re} c_n e^{2\pi i \lambda_n x})
$$
where $\ell_{n+1}= \lambda_{n+1}/\lambda_n$. The assumption 
\eqref{cond-gn} is easily verified, using that $\sup_{n\in \N}|c_n|<1$. 
\hfill $\square$

 \bigskip

 \noindent {\bf Acknowledgement.} The authors are very grateful to
 Istv\'an Berkes for useful discussions and relevant references
 concerning dilated series. They are also thankful to Michel Weber
 for valuable discussions on the same topic.

\end{document}